\documentclass{amsart}

\usepackage{amsmath, amssymb, amsthm, amsfonts}
\usepackage{mathrsfs}

\input xy
\xyoption{all}

\usepackage{hyperref}

\swapnumbers
\numberwithin{equation}{section}

\theoremstyle{plain}
\newtheorem{theorem}[subsection]{Theorem}
\newtheorem{lemma}[subsection]{Lemma}
\newtheorem{prop}[subsection]{Proposition}

\theoremstyle{definition}
\newtheorem{defn}[subsection]{Definition}

\newtheorem{remark}[subsection]{Remark}

\def\FF{\mathbb{F}}
\def\GG{\mathbb{G}}

\def\PP{\mathbb{P}}
\def\QQ{\mathbb{Q}}

\newcommand\cA{\mathcal{A}}
\newcommand\cB{\mathcal{B}}

\newcommand\cE{\mathcal{E}}
\newcommand\cF{\mathcal{F}}
\newcommand\cG{\mathcal{G}}
\newcommand\cH{\mathcal{H}}

\newcommand\cM{\mathcal{M}}

\newcommand\cO{\mathcal{O}}

\newcommand\cR{\mathcal{R}}

\newcommand\cV{\mathcal{V}}
\newcommand\cW{\mathcal{W}}

\newcommand\cZ{\mathcal{Z}}

\def\bL{\mathbf{L}}

\def\bR{\mathbf{R}}

\newcommand{\sG}{\mathscr{G}}

\newcommand{\Bun}{\textup{Bun}}

\newcommand\ev{\textup{ev}}

\newcommand{\Fl}{\textup{Fl}}

\newcommand\Forg{\textup{Forg}}
\newcommand\Fr{\textup{Fr}}
\newcommand\Frac{\textup{Frac}}

\newcommand{\Hk}{\textup{Hk}}

\newcommand\id{\textup{id}}

\newcommand\pr{\textup{pr}}

\newcommand\Proj{\textup{Proj}}

\newcommand\Rep{\textup{Rep}}

\newcommand\Sht{\textup{Sht}}

\newcommand\Spec{\textup{Spec}\ }

\newcommand\Sym{\textup{Sym}}

\newcommand{\univ}{\textup{univ}}

\newcommand\Aut{\textup{Aut}}

\newcommand\uAut{\underline{\Aut}}

\newcommand\GL{\textup{GL}}

\newcommand\PGL{\textup{PGL}}
\newcommand\Ug{\textup{U}}

\newcommand{\Gm}{\GG_m}

\newcommand{\incl}{\hookrightarrow}
\newcommand{\isom}{\stackrel{\sim}{\to}}

\newcommand{\surj}{\twoheadrightarrow}

\newcommand{\twtimes}[1]{\stackrel{#1}{\times}}

\newcommand{\wh}[1]{\widehat{#1}}
\newcommand\quash[1]{}
\newcommand\un{\underline}
\newcommand{\bu}{\bullet}
\newcommand{\ov}{\overline}
\newcommand{\bs}{\backslash}

\newcommand\xr{\xrightarrow}

\newcommand\ot{\otimes}

\newcommand{\sslash}{\mathbin{/\mkern-6mu/}}

\newcommand\upH{\textup{H}}

\renewcommand\a\alpha
\renewcommand\b\beta
\newcommand\g\gamma
\newcommand\G\Gamma
\renewcommand\d\delta

\renewcommand{\th}{\theta}

\newcommand{\ph}{\varphi}
\newcommand{\io}{\iota}
\newcommand{\s}{\sigma}
\renewcommand{\t}{\tau}

\newcommand{\y}{\eta}

\renewcommand{\l}{\lambda}

\newcommand{\om}{\omega}
\newcommand{\Om}{\Omega}

\renewcommand\c{\circ}
\newcommand\To{\Rightarrow}
\newcommand\bt{\boxtimes}

\title{Special cycles for Shtukas are closed}

\dedicatory{To Dick Gross with admiration and gratitude}
\author{Zhiwei Yun}
\thanks{Supported by the Simons Investigatorship and the Packard Fellowship.}

\address{Department of Mathematics, Massachusetts Institute of Technology, 77 Massachusetts Ave, Cambridge, MA 02139}
\email{zyun@mit.edu}
\date{}

\begin{document}

\begin{abstract}
In this paper we give a different proof of a theorem of Paul Breutmann:  for a Bruhat-Tits group scheme $\cH$ over a smooth projective curve $X$ and a closed embedding into another smooth affine group scheme $\cG$, the induced map on the moduli of Shtukas $\Sht^{r}_{\cH}\to \Sht^{r}_{\cG}$ is schematic, finite and unramified. This result enables one to define special cycles on the moduli stack of Shtukas.
\end{abstract}

\maketitle

\section{Introduction}

\subsection{Motivation}\label{ss:mot}

Special cycles on Shimura varieties provide an important link between geometric invariants (such as intersection numbers or heights) and analytic invariants (such as special values or derivatives of $L$-functions).  For a suitable reductive group $G$ over $\QQ$, if we think of (the complex points of) $G$-Shimura varieties as the moduli space of $G$-Hodge structures, then special cycles parametrize $G$-Hodge structures with extra data such as a collection of Hodge tensors, or a reduction to a smaller group $H$.

Fix a smooth projective geometrically connected curve $X$ over $\FF_{q}$. For the function field $F=\FF_{q}(X)$, the role of Shimura varieties is played by the moduli stack of Drinfeld Shtukas over $X$ with $G$-structures. The moduli of $G$-Shtukas $\Sht^{r}_{G}$ has an extra degree of freedom, the number of legs $r$ (equal to $1$ for Shimura varieties).

One can similarly define special cycles on such moduli stacks as the moduli stack of $G$-Shtukas with extra structure such as a reduction to a smaller group $H$ or a Frobenius-invariant section of an associated bundle.  Here are some examples of special cycles for function fields:
\begin{enumerate}
\item Let $\nu: X'\to X$ be an  \'etale double cover, $H=(R_{X'/X}\Gm)/\Gm$ and $G=\PGL_{2}$. In this case, $\Sht^{r}_{H}\to \Sht^{r}_{G}$ are called {\em Heegner-Drinfeld cycles}. Their intersection-theoretic properties were studied in \cite{YZ1,YZ2,HS}, and were given by the $r$th central derivative of the automorphic $L$-functions for $\PGL_{2}$.

%\item Let $\nu: X'\to X$ be an  \'etale double covering, $\cH=R_{X'/X}\GL_{n}$ and $\cG=\GL_{2n}$ (constant). The map $\Bun_{\cH}=\Bun_{n,X'}\to \Bun_{\cG}=\Bun_{2n,X}$ maps a rank $n$ vector bundle $\cG$ on $X'$ to its direct image $\cE=\nu_{*}\cF$, which has rank $2n$ on $X$. When $n=1$ this is essentially the situation studied in \cite{YZ1}. The classical counterpart of this situation ($\GL_{n}(E)$-periods of $\GL_{2n}(F)$-automorphic forms, where $E/F$ is quadratic) for general $n$ is studied by J.Guo and Jacquet using relative trace formula \cite{GJ}. 

\item Let $X'\to X$ be a double cover and $n\ge1$. Consider $\Ug_{n}$ a unitary group scheme over $X$ that splits over $X'$.  Letting $H=\Ug_{n}$ and $G=\Ug_{n}\times\Ug_{n+1}$, we have a map $\Sht^{r}_{H}\to \Sht^{r}_{G}$ coming from  the diagonal embedding $H\incl G$.  The image of this map is the function field analogue of Gan-Gross-Prasad cycles in the arithmetic GGP conjecture \cite{GGP}. 

\item Let $G=\Ug_{n}$ using a double cover  as above, and let $\cE$ be a vector bundle of rank $m\le n$ over $X'$. In \cite{FYZ, FYZ2} we define a special cycle $\cZ^{r}_{\cE}$ on $\Sht^{r}_{G}$ parametrizing unitary Shtukas with a Frobenius invariant map from $\cE$. These are function field analogues of the special cycles defined by Kudla and Rapoport \cite{KR1,KR2}. In the case $m=n$, they are $0$-cycles and their degrees are given by the Fourier coefficients of the $r$th central derivative of Siegel-Eisenstein series \cite{FYZ}.
\end{enumerate}

A basic question, before one can even call special cycles ``cycles'', is to show that their image in $\Sht^{r}_{G}$ is closed. In the above examples, the closedness is easy to show using special features of the situation (for example in Example (1) $\Sht^{r}_{H}$ is itself proper). We ask more generally if $\cH\incl \cG$ is a closed embedding  of smooth affine group schemes over $X$, is the induced map
$$\th^{r}: \Sht^{r}_{\cH}\to\Sht^{r}_{\cG}$$
a finite map? The non-obvious part is the properness of the map.
 
 \subsection{Result and proof outline}
In \cite[Theorem 3.26]{Br}, P. Breutmann proves that when $\cH$ is a Bruhat-Tits group scheme (i.e. connected reductive generically with parahoric level structures), $\th^{r}$ is schematic,  finite and unramified. This allows one to define special cycles on $\Sht^{r}_{\cG}$ as the direct image of the fundamental class of $\Sht^{r}_{\cH}$.

In this paper we give a different (and hopefully more streamlined) proof of this fact. In the publicly available version of \cite{Br}, only a weaker statement was proved ($\th^{r}$ has the same properties after completion along a fixed leg). Our original purpose was to improve that result over all of $X^{r}$. After this paper was written, we learned from Urs Hartl that a revised version of \cite{Br} (not posted on the arXiv yet) contains the statement that $\th^{r}$ is schematic,  finite and unramified.

Our proof consists of two parts. In the first part we construct a closed embedding of $\Bun_{\cH}$ to the moduli stack $\Bun_{\cG}(\cW)^{\c}$ of $\cG$-torsors together with a nonzero section of the vector bundle associated to a certain $\cG$-representation $\cW$. The argument is easy when $\cH$ is reductive everywhere, but when $\cH$ has parahoric level structures it relies on a deep result of Ansch\"utz \cite{An} (as does Breutmann's argument).

In the second part we pass to Shtuka versions of the moduli stacks considered in the first part. The main observation is that the forgetful map $\Sht^{r}_{\cG}(\cW)^{\c}\to \Sht^{r}_{\cG}$ is proper. This argument was used in \cite{FYZ} to show that Kudla-Rapoport special cycles are closed.

We hope Breutmann's result and this short paper will provide a starting point for the study of more examples of special cycles on the moduli of Shtukas.  We refer to \cite{Y-ICM} for a survey of what was known and what is expected about special cycles  (called Heegner-Drinfeld cycles in {\em loc.cit.}) .

%%% thanks %%%

\subsection*{Acknowledgement} 
I would like to thank Brian Conrad for help on reductive group schemes, and Urs Hartl for sending me the revised version of Breutmann's paper \cite{Br}. I would also like to thank the referee for helpful suggestions.

It is a pleasure to dedicate this short paper to Dick Gross. His original results and conjectures on Heegner points, diagonal cycles on product of modular curves, and diagonal cycles on unitary Shimura varieties are our main inspiration for studying special cycles over function fields. At a personal level, I would like to thank Dick for his guidance and encouragement over the last 12 years.

\section{Definitions and statements}
Let $X$ be a smooth, projective and geometrically connected curve over $k=\FF_{q}$. Let $\cG$ be a smooth affine group scheme over $X$.

\subsection{Moduli of bundles}
Let $\Bun_{\cG}$ denote the moduli stack of $\cG$-torsors over $X$. More precisely, for any affine scheme $S$, $\Bun_{\cG}(S)$ is the groupoid of $\cG$-torsors over $X\times S$.

Let $\Rep(\cG)$ be the category of (finite rank) vector bundles over $X$ with a linear action of $\cG$. Morphisms in  $\Rep(\cG)$  are $\cG$-equivariant linear maps.

For  $\cV\in \Rep(\cG)$ and a $\cG$-torsor $\cE$ over $X$, let $\cV_{\cE}$ be the associated vector bundle $\cE\times^{\cG}_{X}\cV$, the quotient of $\cE\times_{X}\cV$ by the diagonal action of $\cG$.

Let $r\ge0$ be an integer.  Let $\Hk^{r}_{\cG}$ be the ind-stack whose $S$-points ($S$ affine) classify tuples $(x_{\bu}=\{x_{i}\}_{1\le i\le r}, \cE_{\bu}=\{\cE_{i}\}_{0\le i\le r}, f_{\bu}=\{f_{i}\}_{1\le i\le r})$, where
\begin{itemize}
\item $x_{i}: S\to X$ ($1\le i\le r$) are morphisms called {\em legs};
\item $\cE_{i}$ ($0\le i\le r$) are $\cG$-torsors over $X\times S$;
\item $f_{i}: \cE_{i-1}|_{X\times S\bs\G(x_{i})}\isom \cE_{i}|_{X\times S\bs\G(x_{i})}$ ($1\le i\le r$) are isomorphisms of $\cG$-torsors. 
%with bounded zeros and poles along the divisor $\G(x_{i})$. 
\end{itemize}
%Here, having  bounded zeros and poles means that for a faithful linear representation $\cV\in\Rep(\cG)$, the induced isomorphism $\cV_{\cE_{i-1}}|_{X\times S\bs\G(x_{i})}\isom \cV_{\cE_{i}}|_{X\times S\bs\G(x_{i})}$ has finite pole order along $\G(x_{i})$). 

One can specifying the bounds $\mu=(\mu_{1},\cdots, \mu_{r})$ on the zeros and poles of $(f_{1},\cdots,f_{r})$ to get algebraic substacks $\Hk^{r,\le \mu}_{\cG}$ of $\Hk^{r}_{\cG}$, so that $\Hk^{r}_{\cG}$ is the inductive limit of $\Hk^{r,\le \mu}_{\cG}$, with transition maps closed embeddings.  Our result will not be sensitive to the precise definition of these bounds, therefore we will always work with the ind-stack $\Hk^{r}_{\cG}$. When $\cG$ is split, one can specify a bound $\mu$  to be a tuple $(\mu_{1},\cdots, \mu_{r})$ of dominant coweights of $\cG$. We refer to the paper of Arasteh Rad and Hartl \cite{AH} for more precise discussion about bounds.

\subsection{Moduli of Shtukas}  \label{ss:def Sht}
Shtukas for $\GL_{n}$ are introduced by Drinfeld \cite{Dr}. For split reductive groups $G$,  $G$-Shtukas are  defined by Varshavsky \cite{Var}. The general notion of $\cG$-Shtukas for group schemes $\cG$ are introduced and studied by Arasteh Rad and Hartl \cite{AH}.

Let $r\ge0$ be an integer. We define $\Sht_{\cG}^{r}$ by the Cartesian diagram
\begin{equation}\label{def Sht}
\xymatrix{\Sht^{r}_{\cG}\ar[d] \ar[r] & \Hk^{r}_{\cG}\ar[d]^{(p_{0},p_{r})}\\
\Bun_{\cG}\ar[r]^-{(\id, \Fr)} &  \Bun_{\cG}\times \Bun_{\cG}
}
\end{equation}
Here $p_{i}: \Hk^{r}_{\cG}\to \Bun_{\cG}$ is the forgetful map recording the $\cG$-torsor $\cE_{i}$, for $0\le i\le r$.

%Let $\mu=(\mu_{1},\cdots, \mu_{r})$ be a sequence of dominant coweights of a maximal split torus $T$ of $G$. Recall dominant coweights of $T$ are in bijection with relative positions of two $G$-bundles over the formal disk with the same generic fiber. 
%
%
%Let $\Hk^{\mu}_{G}$ be the Hecke stack whose $S$-points are tuples $(x_{\bu}=\{x_{i}\}_{1\le i\le r}, \cE_{\bu}=\{\cE_{i}\}_{0\le i\le r}, f_{\bu}=\{f_{i}\}_{1\le i\le r})$, where
%\begin{itemize}
%\item For $1\le i\le r$, $x_{i}: S\to  X$ which graph denoted by $\G(x_{i})\subset X\times S$.
%\item For $0\le i\le r$, $\cE_{i}$ are $G$-torsors over $X\times S$ with $K$-level structures.
%\item For $1\le i\le r$, $f_{i}: \cE_{i-1}|_{X\times S\bs\G(x_{i})}\isom \cE_{i}|_{X\times S\bs\G(x_{i})}$ is an isomorphism of $G$-torsors  compatible with the level structures, whose relative position at $x_{i}$ is bounded by $\mu_{i}$.
%\end{itemize}

Thus an $S$-point of $\Sht^{r}_{\cG}$ is  a tuple $(x_{\bu}, \cE_{\bu}, f_{\bu}, \io)$, where
\begin{itemize}
\item $(x_{\bu}, \cE_{\bu}, f_{\bu})\in  \Hk^{r}_{\cG}(S)$;
\item $\io:\cE_{r}\isom {}^{\tau}\cE_{0}$ is an isomorphism of $\cG$-torsors over $X\times S$. Here, ${}^{\tau}\cE_{0}=(\id_{X}\times\Fr_{S})^{*}\cE_{0}$.  
\end{itemize}
We call such a tuple $(x_{\bu}, \cE_{\bu}, f_{\bu}, \io)$ an $S$-family of {\em $\cG$-Shtukas with $r$ legs}. 

%Recording the points $x_{\bu}$ gives a morphism
%\begin{equation*}
%\pi^{r}_{\cG}: \Sht^{r}_{\cG}\to X^{r}
%\end{equation*}

Using a bound $\mu$ and $\Hk^{r,\le \mu}_{\cG}$ instead of $\Hk^{r}_{\cG}$, a similar diagram as \eqref{def Sht} defines  $\Sht^{r,\le \mu}_{\cG}$ which are algebraic stacks locally of finite type over $k$. We see that $\Sht^{r}_{\cG}$ is the inductive limit of $\Sht^{r,\le \mu}_{\cG}$ with transition maps closed embeddings. 

\subsection{Special cycles} Let $\cH$ and $\cG$ be smooth affine group schemes over $X$, and $\cH\incl \cG$ be a closed embedding.  

%Choose maximal tori $T_{H}$ and $T$ of $H$ and $G$ respectively such that $T_{H}\subset T$. Choose Borel subgroups $B_{H}$ and  $B$ of $G$ respectively such that $T_{H}\subset B_{H}\subset B\supset T$.
%
%Let $\xcoch(T_{H})^{+}$ be the set of dominant coweights for $T_{H}$ with respect to $B_{H}$; similarly define $\xcoch(T)^{+}$. Let $\l\in \xcoch(T_{H})^{+}$ and $\mu\in \xcoch(T)^{+}$. Denote $\l^{+}\in \xcoch(T)^{+}$ the unique dominant coweight in the $W$-orbit of $\l$ (viewed as an element of $\xcoch(T)$).  We say $\l\le_{G}\mu$ if $\mu-\l^{+}$ is a zero or a sum of positive coroots of $G$.
%
%
%Let $\l_{1},\cdots, \l_{r}\in \xcoch(T_{H}^{+})$ and  $\mu_{1},\cdots, \mu_{r}\in \xcoch(T)^{+}$ satisfying $\l_{i}\le_{G}\mu_{i}$ for $1\le i\le r$. 

The inclusion  $\cH\incl \cG$ induces a map $\ph: \Bun_{\cH}\to \Bun_{\cG}$ sending a $\cH$-torsor $\cF$ to the induced $\cG$-torsor $\cE=\cF\twtimes{\cH}_{X}\cG$. The same construction gives a map of ind-stacks $h^{r}: \Hk^{r}_{\cH}\to \Hk^{r}_{\cG}$.  Using the Cartesian diagram \eqref{def Sht}, the map $h$ induces a map of ind-stacks
\begin{equation*}
\th^{r}: \Sht^{r}_{\cH}\to \Sht^{r}_{\cG}.
\end{equation*}
For any bound $\l$ of modifications for $\cH$-torsors, $\th^{r}$ sends $\Sht^{r,\le \l}_{\cH}$ to $\Sht^{r,\le \mu}_{\cG}$ for some bound $\mu$ of modifications for $\cG$-torsors. Conversely, for any bound $\mu$ of Hecke modifications for $\cG$-torsors, its preimage $\th^{r,-1}(\Sht^{r,\le \mu}_{\cG})$ is contained in a finite union of $\Sht^{r,\le\l'}_{\cH}$  for bounds $\l'$ of Hecke modifications for $\cH$-torsors.

\subsection{Bruhat-Tits group scheme}\label{ss:BT} Let $\cH$ be a smooth affine group scheme over $X$. We say $\cH$ is a {\em Bruhat-Tits group scheme} if
\begin{enumerate}
\item The generic fiber of $\cH$ is a connected reductive group over the function field $F=k(X)$. This implies that there is an open dense subset $U\subset X$ over which $\cH$ is a connected reductive group scheme.
\item For each $x\in X\bs U$, writing $\cO_{x}$ for the completed local ring at $x$ and $F_{x}$ its fraction field, the group scheme $\cH|_{\Spec \cO_{x}}$ is a parahoric subgroup of $\cH|_{\Spec F_{x}}$.\end{enumerate}

In this paper we will give a different proof of the following theorem of Breutmann.

\begin{theorem}[Breutmann {\cite[Theorem 3.26]{Br}}]\label{th:main} Let $\cH\incl\cG$ be a closed embedding of smooth affine group schemes over $X$. Assume $\cH$ is a Bruhat-Tits group scheme over $X$. Then the map $\th^{r}: \Sht^{r}_{\cH}\to \Sht^{r}_{\cG}$ is schematic, finite and unramified.
\end{theorem}
Concretely, this means that for any bound $\mu$ of modifications for $\cG$-torsors, the restriction  $\th^{r,-1}(\Sht^{r,\le \mu}_{\cG})\to \Sht^{r,\le \mu}_{\cG}$, as a map of algebraic stacks, is schematic, finite and unramified.

Following this theorem, one can define special cycles as follows. Choose bounds $\l$ and $\mu$ for $\cH$ and $\cG$ such that $\th^{r}$ restricts to a map $\Sht^{r,\le \l}_{\cH}\to \Sht^{r,\le \mu}_{\cG}$. Assume $\Sht^{r,\le \l}_{\cH}$ has a fundamental cycle $[\Sht^{r,\le \l}_{\cH}]$ (for example, if $\Sht^{r,\le \l}_{\cH}$ is smooth). Then we get a cycle $\th^{r}_{\mu,*}[\Sht^{r,\le \l}_{\cH}]$ on $\Sht^{r,\le \mu}_{\cG}$ by pushing forward along $\th^{r}$. We would like to call such algebraic cycles on $\Sht^{r,\le \mu}_{\cG}$ {\em special cycles}.

\subsection{Pseudo-homomorphisms} We give a slight generalization of the above theorem to allow more flexibility. 

A subtlely is that maps between the moduli stacks of bundles $\Bun_{\cH}\to \Bun_{\cG}$ does not necessarily come from a homomorphism $\cH\to \cG$. This can already be seen in the first example in \S\ref{ss:mot}.

\begin{defn} Let $\cH$ and $\cG$ be group schemes over $X$. A pseudo-homomorphism $\cE_0: \cH\To\cG$ is a right $\cG$-torsor $\cE_0$ over $X$ with a commuting (left) action of $\cH$. 
\end{defn}

%We remark that if $\cH$ is a constant group scheme $\cH=H\times X$, then any map $\ph: \Bun_{H,X}\to \Bun_{\cG,X}$ comes from a pseudo-homomorphism $\cE_0:\cH\To \cG$, unique up to isomorphism. Indeed, the image of the trivial $H$-torsor under $\ph$ is a $\cG$-torsor $\cE_0$ over $X$ with a commuting $H$-action, which is the same as a commuting $\cH$-action, so that $\ph=\ph_{\cE_0}$. 

A pseudo-homomorphism  $\cE_0: \cH\To \cG$ induces a homomorphism of group schemes $i_{\cE_0}: \cH\to \uAut_{\cG}(\cE_0)$  (the latter being an inner form of $\cG$ over $X$). Since $\Bun_{\cG}$ is unchanged if $\cG$ is replaced by an inner form, $i_{\cE_{0}}$ induces a map $\ph_{\cE_{0}}: \Bun_{\cH}\to \Bun_{\cG}$. More precisely, for $\cF\in \Bun_{\cH}(S)$, $\ph_{\cE_0}(\cF)=\cF\twtimes{\cH}_{X}\cE_0$ which is a right $\cG$-torsor over $X\times S$.  

\begin{defn}
We say a pseudo-homomorphism  $\cE_0: \cH\To \cG$ is a {\em pseudo-closed embedding} if the induced map $i_{\cE_0}: \cH\to \uAut_{\cG}(\cE_0)$ is a closed embedding.
\end{defn}

\begin{remark}
Let $B\cG$ be the classifying stack of $\cG$ over $X$. Namely, for any test scheme $S$ with a map $x:S\to X$, $B\cG(S)$ is the groupoid of $x^{*}\cG$-torsors over $S$. This is a gerbe over $X$. Then a pseudo-homomorphism $\cE_0: \cH\To\cG$ is the same datum as a morphism of stacks $\b: B\cH\to B\cG$. Namely, given $\b$, the composition $X\to B\cH\xr{\b}B\cG$ gives a right $\cG$-torsor $\cE_0$ over $X$ with a commuting action of $\cH$. Conversely, given a pseudo-homomorphism $\cE_0: \cH\To\cG$, for any scheme $S$ with $x:  S\to X$ and any $x^{*}\cH$-torsor $\cF_{x}$ over $S$, we define $\b_S\in B\cG(S)$ (a $x^{*}\cG$-torsor over $S$) to be $\cF_{x}\times^{x^{*}\cH}_{S}x^{*}\cE_0$.

In some cases, it is more natural to define the moduli stack of torsors not starting from a group scheme $\cG$ over $X$, but starting from a gerbe $\sG$ over $X$. See \cite[\S3.1]{FYZ2}.
\end{remark}

A pseudo-homomorphism $\cE_{0}: \cH\To\cG$  similarly induces a map $h^{r}_{\cE_{0}}: \Hk^{r}_{\cH}\to \Hk^{r}_{\cG}$ covering $\ph_{\cE_{0}}: \Bun_{\cH}\to \Bun_{\cG}$ via the projections $p_{i}: \Hk^{r}_{\cH}\to \Bun_{\cH}$ and $p_{i}: \Hk^{r}_{\cG}\to \Bun_{\cG}$, $0\le i\le r$. Therefore it also induces a map of ind-stacks
\begin{equation*}
\th^{r}_{\cE_{0}}: \Sht^{r}_{\cH}\to \Sht^{r}_{\cG}.
\end{equation*}

\begin{theorem}\label{th:Sht gp sch} Let $\cE_{0}: \cH\To\cG$ be a pseudo-closed embedding of smooth affine group schemes  over $X$. Assume $\cH$ is a Bruhat-Tits group scheme over $X$. Then the map $\th^{r}_{\cE_{0}}: \Sht^{r}_{\cH}\to \Sht^{r}_{\cG}$ is schematic, finite and unramified.
\end{theorem}
This theorem directly follows from Theorem \ref{th:main}: we may replace $\cG$ with its inner form $ \uAut_{\cG}(\cE_0)$ without changing the moduli stack $\Sht^{r}_{\cG}$, and reduce to the situation of a closed embedding $\cH\incl\uAut_{\cG}(\cE_0)$.

The rest of the paper is devoted to the proof of Theorem \ref{th:main}.

\section{Reductions of bundles}
In this section, $k$ is an arbitrary field.   Let $X$ be a smooth, projective and geometrically connected curve over $k$. Let $\cH\incl \cG$ be a closed embedding of smooth affine group schemes over $X$.

\begin{remark}\label{r:enough V}
A typical way to construct objects $\cV\in \Rep(\cG)$ is as follows. Let $\cG_{\y}$ be the generic fiber of $\cG$, a reductive group over $F=k(X)$. Now the coordinate ring $\cO_{\cG_{\y}}=F[\cG_{\y}]$ is a union of finite-dimensional $F$-subspaces $V_{i}$ stable under the right translation of $\cG_{\y}$. Let $\cV_{i}$ be the saturation of $V_{i}$ in $\cO_{\cG}$ (which is a union of vector bundles on $X$), then $\cV_{i}\in\Rep(\cG)$ with the $\cG$-action inherited from the right translation action on $\cO_{\cG}$, and $\cO_{\cG}$ is the union of $\cV_{i}$.

More generally, suppose $\cR$ is a quasi-coherent sheaf on $X$ with a linear $\cG$-action (i.e., an $\cO_{\cG}$-comodule), and $\cR$ is a union of vector bundles, then $\cR$ is a union of $\cG$-stable vector bundles (i.e., objects in $\Rep(\cG)$). This can be checked by the saturation as above.
\end{remark}

The stack quotient $\cG/\cH$ is an algebraic space. Let $\cV\in \Rep(\cG)$. Let $\cV^{\cH}$ be the subbundle of $\cH$-invariants. The action of $\cG$ on $\cV$ moves the natural embedding $\cV^{\cH}\incl \cV$ and gives a morphism $\cG\to (\cV^{\cH})^{\vee}\ot_{\cO_{X}}\cV$ that is right invariant under $\cH$. It therefore induces a unique map of algebraic spaces over $X$
\begin{equation}\label{bv}
b_{\cV}: \cG/\cH\to (\cV^{\cH})^{\vee}\ot_{\cO_{X}} \cV.
\end{equation}
Here the right side is identified with the total space of the vector bundle with the same name.

\subsection{The case $\cH$ is reductive}
From now until \S\ref{ss:Hpar}, we assume $\cH$ is a reductive group scheme over $X$. In this case, the fppf sheaf $\cG/\cH$ is representable by a scheme which is affine over $X$. This is a special case of a general theorem of Alper \cite[Theorem 9.4.1]{Alper} (which only requires $\cG$ to be affine over $X$). This immediately implies that the natural map
\begin{equation*}
\cG/\cH\to \cG\sslash \cH=\un{\textup{Spec}}_{X}(\cO_{\cG}^{\cH})
\end{equation*}
is an isomorphism.

%Passing to the quotient by $\cG$, it gives the map
%\begin{equation*}
%b_{\cV}: \BB\cH\to \cG\bs ((\cV^{\cH})^{\vee}\ot \cV)
%\end{equation*}
%where in the latter we identify $(\cV^{\cH})^{\vee}\ot \cV$ with the total space of the vector bundle over $X$, and $\cG$ only acts on the $\cV$-factor.

\begin{lemma}\label{l:GH gen sch} Assume $\cH$ is reductive. There exists   $\cV\in\Rep(\cG)$ such that $b_{\cV}$ is a closed embedding.
\end{lemma}
\begin{proof} %Let $\cG_{\y}$ be the generic fiber of $\cG$, a reductive group over $F=k(X)$. Now  $\cO_{\cG_{\y}}=F[\cG_{\y}]$ is a union of finite-dimensional $F$-subspaces $V_{i}$ stable under the right translation of $\cG_{\y}$. Let $\cV_{i}$ be the saturationof $V_{i}$ in $\cO_{\cG}$ (which is a union of coherent subsheaves), then $\cV_{i}$ is a finite rank vector bundle over $X$ stable under the right translation of $\cG$.

Since $\cO_{\cG}$ is a finitely generated sheaf of algebras over $\cO_{X}$, by Remark \ref{r:enough V}, there exists a $\cG$-stable subbundle $\cV_{0}\subset \cO_{\cG}$ such that $\Sym(\cV_{0})\to \cO_{\cG}$ is surjective.  In other words, $\cG$ embeds into the total space of the vector bundle $\cV_{0}^{\vee}$. This allows us to apply  Seshadri's theorem \cite[Theorem 3]{Ses} to conclude that $\cO_{\cG}^{\cH}$ is a finitely generated sheaf of algebras over $\cO_{X}$.  Choosing a $\cG$-stable $\cV\subset \cO_{\cG}$ such that $\cV^{\cH}$ generates $\cO_{\cG}^{\cH}$ as a quasi-coherent sheaf of algebras. We claim that $b_{\cV}$ is a closed embedding. To see this, it suffices to show that the image of $m_{\cV}: \cV^{\cH}\ot\cV^{\vee}\to \cO_{\cG}^{\cH}$ contains $\cV^{\cH}$. But let $e:\cO_{X}\to \cV^{\vee}$ correspond to the map $\cV\subset \cO_{\cG}\xr{\ev_{1}}\cO_{X}$ ($\ev_{1}$ is restriction to the unit section), then the composition
\begin{equation*}
\cV^{\cH}\xr{\id\ot e }\cV^{\cH}\ot\cV^{\vee}\xr{m_{\cV}} \cO_{\cG}^{\cH}
\end{equation*}
is the inclusion of $\cV^{\cH}$ in $\cO_{\cG}^{\cH}$. Therefore the image of $m_{\cV}$ contains algebra generators of $\cO_{\cG}^{\cH}$, and $b_{\cV}$ is a closed embedding.
\end{proof}

\subsection{Bundles with sections}\label{ss:bV Bun} 
Let $Z\to X$ be a stack with a left $\cG$-action. For any right $\cG$-torsor $\cE$ over $X\times S$, we can form the twisted product $\pi_{\cE,Z}: \cE\times^{\cG}_{X}Z\to X\times S$, which \'etale locally over $X\times S$ is isomorphic to $Z\times S$. Let $\Bun_{\cG}(Z)$ be the stack whose $S$-points consists of pairts $(\cE, t)$ where $\cE$ is a right $G$-torsor over $X\times S$ and $t: X\times S\to \cE\times^{\cG}_{X}Z
$ be a section of $\pi_{\cE,Z}$.

When $Z$ is the total space of  $\cW\in \Rep(\cG)$, then we write $\Bun_{\cG}(\cW)$ for $\Bun_{\cG}(Z)$. In this case, $\Bun_{\cG}(\cW)$ is the relative spectrum of the symmetric algebra of a coherent sheaf (this coherent sheaf is $\cB=R^{1}p_{\Bun*}(\cW^{\vee}_{\univ}\ot p^{*}_{X}\om_{X})$, where $p_{\Bun}: X\times \Bun_{\cG}\to \Bun_{\cG}$ and $p_{X}: X\times \Bun_{\cG}\to X$ are the projections, and $\cW_{\univ}$ is the bundle associated to $\cW$ and the universal $\cG$-torsor over $X\times \Bun_{\cG}$).

%Let $\cV\in\Rep(\cG)$. Let $\cF$ be another vector bundle on $X$. Let $\Bun_{\cG}(\cV;\cF)$ be the moduli stack  whose $S$-points are pairs  $(\cE, t)$ where $\cE$ is a right $G$-torsor over $X\times S$ and $t: \cF\ot \cO_{X\times S}\to \cV_{\cE}=\cE\times^{\cG}_{X}\cV$ is a map of coherent sheaves on $X\times S$. We remark that $\Bun_{\cG}(\cV;\cF)$ can be represented by the relative spectrum over $\Bun_{\cG}$ of the symmetric algebra of a coherent sheaf (this coherent sheaf is $\cB=R^{1}p_{\Bun*}(\cF\ot \cV^{\vee}_{\univ}\ot p^{*}_{X}\om_{X})$, where $p_{\Bun}: X\times \Bun_{\cG}\to \Bun_{\cG}$ and $p_{X}: X\times \Bun_{\cG}\to X$ be the projections, and $\cV_{\univ}$ is the bundle associated to $\cV$ and the universal $\cG$-torsor over $X\times \Bun_{\cG}$).  

It is clear that $\Bun_{\cG}(\cG/\cH)\cong \Bun_{\cH}$.  Applying the map $b_{\cV}$ to the construction $\Bun_{\cG}(-)$ gives a map of stacks
\begin{equation*}
b^{\Bun}_{\cV}: \Bun_{\cH}\cong \Bun_{\cG}(\cG/\cH)\to \Bun_{\cG}((\cV^{\cH})^{\vee}\ot \cV).
\end{equation*}

\begin{lemma}\label{l:BunH emb sch} Assume $\cH$ is reductive, and let $\cV$ be as in Lemma \ref{l:GH gen sch}.
\begin{enumerate}
\item The map $b^{\Bun}_{\cV}$ is a closed embedding.
\item Let $\Bun_{\cG}((\cV^{\cH})^{\vee}\ot \cV)^{\c}\subset \Bun_{\cG}((\cV^{\cH})^{\vee}\ot \cV)$ be the open substack consisting of $(\cE,t)$ such that $t|_{X\times \{s\}}$ is nonzero for any geometric point $s\in S$. Then the image of $b^{\Bun}_{\cV}$ lies in $\Bun_{\cG}((\cV^{\cH})^{\vee}\ot \cV)^{\c}$. In particular,  $b^{\Bun}_{\cV}$ induces a closed embedding
\begin{equation*}
b^{\Bun,\c}_{\cV}: \Bun_{\cH}\incl \Bun_{\cG}((\cV^{\cH})^{\vee}\ot \cV)^{\c}.
\end{equation*}
\end{enumerate}
\end{lemma}
\begin{proof}
(1) We first observe that if $Z_{1}\incl Z_{2}$ is a closed embedding of $\cG$-schemes over $X$, then the induced map $\Bun_{\cG}(Z_{1})\to \Bun_{\cG}(Z_{2})$ is also a closed embedding. This is easy to see by making a base change to any test scheme mapping to $\Bun_{\cG}(Z_{2})$. Applying this observation to the closed embedding $b_{\cV}$, we conclude that $\Bun_{\cH}\cong \Bun_{\cG}(\cG/\cH)\to \Bun_{\cG}((\cV^{\cH})^{\vee}\ot \cV)$ is  a closed embedding.

(2) If $\cF\in \Bun_{\cH}(S)$ is an $\cH$-torsor over $X\times S$, and $\cE=\cF\times^{\cH}\cG\in \Bun_{\cG}(S)$, the section $t$ of $(\cV^{\cH})^{\vee}\ot_{\cO_{X}}\cV_{\cE}$ is induced from the map
\begin{equation*}
\cV^{\cH}\bt \cO_{S}=(\cV^{\cH})_{\cF}\incl \cV_{\cF}=\cV_{\cE}
\end{equation*}
which is clearly nonzero along $X\times \{s\}$ for any geometric point $s\in S$. Hence the image of $b^{\Bun}_{\cV}$ lands in $\Bun_{\cG}((\cV^{\cH})^{\vee}\ot \cV)^{\c}$.
\end{proof}

\subsection{The case  $\cH$ is a Bruhat-Tits group scheme}\label{ss:Hpar}
For the rest of the section, we assume $\cH$ is a Bruhat-Tits group scheme over $X$, see \S\ref{ss:BT}.

For $\cV\in\Rep(\cG)$, consider the map $b_{\cV}$ defined in \eqref{bv}. We can choose $\cV\in\Rep(\cG)$ such that $b_{\cV}$ is a closed embedding when restricted to $U$. This is possible:  by applying Lemma \ref{l:GH gen sch} to the reductive groups $\cH_{U}\incl \cG_{U}$ over $U$ (the proof of Lemma \ref{l:GH gen sch} works for $U$ in place of $X$), we obtain $\cV'\in \Rep(\cG_{U})$ such that $b_{\cV'}: \cG_{U}/\cH_{U}\to (\cV'^{\cH_{U}})^{\vee}\ot_{\cO_{U}}\cV'$ is a closed embedding. By Remark \ref{r:enough V} we may assume $\cV'\subset \cO_{\cG_{U}}$. Then take $\cV\in \Rep(\cG)$ to be the saturation of $\cV'$ over $X$, so that $b_{\cV}|_{U}=b_{\cV'}$ is a closed embedding.

The following result generalizes Lemma \ref{l:BunH emb sch} to the case $\cH$ is a Bruhat-Tits group scheme. The essential part of the proof follows the same lines as Step 2 in Breutman's proof of \cite[Theorem 3.26]{Br}, which relies on a deep result of Ansch\"utz \cite{An}.

\begin{prop}\label{l:bV par} Assume $\cH$ is a Bruhat-Tits group scheme over $X$, and that $\cV\in \Rep(\cG)$ is such that $b_{\cV}$ is a closed embedding over $U$. Then the map $b_{\cV}^{\Bun}$ is a closed embedding.
\end{prop}
\begin{proof}
For the proof we can base change the situation to $\ov k$. Therefore we will assume $k$ is algebraically closed.

Let $\cZ\subset (\cV^{\cH})^{\vee}\ot\cV$ be the closure of the image of $\cG_{U}/\cH_{U}$. Since $\cG/\cH$ is smooth over $X$, $\cG_{U}/\cH_{U}$ is dense in $\cG/\cH$ hence $b_{\cV}$ lands in $\cZ$. Therefore we have a $\cG$-equivariant map
\begin{equation}\label{cz}
c_{\cV}: \cG/\cH\to \cZ
\end{equation}
which is an isomorphism over $U$.  It then induces a map 
\begin{equation*}
c^{\Bun}_{\cV}: \Bun_{\cH}\cong \Bun_{\cG}(\cG/\cH)\to \Bun_{\cG}(\cZ)
\end{equation*}
such that $b^{\Bun}_{\cV}$ is the composition of $c^{\Bun}_{\cV}$ followed by the closed embedding $\Bun_{\cG}(\cZ)\incl \Bun_{\cG}((\cV^{\cH})^{\vee}\ot\cV)$ (see the proof of Lemma \ref{l:BunH emb sch}(1)). Therefore it suffices to prove that $c^{\Bun}_{\cV}$ is a closed embedding.

%Since  $b_{\cV}$ is a closed embedding over $U$, we can find a $\cG$-stable vector bundle $\cW_{U}$ (over $U$) inside the defining ideal of $\cG_{U}/\cH_{U}$ in $((\cV^{\cH})^{\vee}\ot\cV)|_{U}$ that generate the ideal. Saturating $\cW_{U}$ inside  $\Sym(\cV^{\cH}\ot\cV^{\vee})$ we get a $\cG$-stable vector bundle $\cW$ over $X$. By construction, we have a (not necessarily linear) $\cG$-equivariant map  $f:  (\cV^{\cH})^{\vee}\ot \cV\to \cW^{\vee}$ such that $b_{\cV}(\cG_{U}/\cH_{U})$ is the preimage of the zero section $0_{\cW^{\vee}}|_{U}$ under $f$. Therefore, $b_{\cV}$ factors as:
%\begin{equation}\label{cz}
%\cG/\cH\xr{c_{\cV}} \cZ:=f^{-1}(0_{\cW^{\vee}})\subset (\cV^{\cH})^{\vee}\ot \cV
%\end{equation}
%and $c_{\cV}$ is an isomorphism over $U$.
%
%The map $f$ induces a map over $\Bun_{\cG}$
%\begin{equation*}
%f^{\Bun}: \Bun_{\cG}((\cV^{\cH})^{\vee}\ot \cV)\to \Bun_{\cG}(\cW^{\vee}).
%\end{equation*}
%Let $i: \Bun_{\cG}\incl \Bun_{\cG}(\cW^{\vee})$ be the closed embedding of the zero section (the datum $t: \cO_{X\times S}\to \cW^{\vee}_{\cE}$ is taken to be zero). Let $\Bun_{\cG}(\cZ)=f^{\Bun,-1}(i(\Bun_{\cG}))$, which is a closed substack of $\Bun_{\cG}((\cV^{\cH})^{\vee}\ot \cV)$. Then the map $b^{\Bun}_{\cV}$ factor through
%\begin{equation*}
%c^{\Bun}_{\cV}: \Bun_{\cH}\to \Bun_{\cG}(\cZ)
%\end{equation*}
%It suffices to prove that $c^{\Bun}_{\cV}$ is a closed embedding. 

First we check that $c^{\Bun}_{\cV}$ is schematic and of finite type. By \cite[Proposition 2.2(a)]{AH}, there is a faithful representation $\cV'$ of $\cH$ such that $\cH\incl \GL(\cV')$ is a closed embedding, $\GL(\cV')/\cH$ is quasi-affine over $X$ and it admits a $\GL(\cV')$-equivariant open embedding into a $\GL(\cV')$-scheme $Y$ affine over $X$. Moreover, the construction of $\cV'$ in {\em loc.cit.} allows us to arrange that  $\cV'$ extends to a representation of $\cG$. Now we have maps 
$$\rho: \Bun_{\cH}\xr{c^{\Bun}_{\cV}} \Bun_{\cG}(\cZ)\to\Bun_{\cG} \to \Bun_{\GL(\cV')}\cong\Bun_{N}$$ 
where $N$ is the rank of $\cV'$. By \cite[Theorem 2.6]{AH} applied to the embedding $\cH\incl \GL(\cV')$, we conclude that $\rho$ is schematic and of finite type.  So {\em a fortiori} $c^{\Bun}_{\cV}$ is schematic and of finite type.

Next we check that $c^{\Bun}_{\cV}$ is proper. For this it suffices to check that $c^{\Bun}_{\cV}$ satisfies the existence and uniqueness of the valuative criterion for DVRs \cite[Lemma 104.11.2, 104.11.3]{Stack}. Let $R$  be a DVR with $K=\Frac(R)$. Let $(\cE,t): \Spec R\to \Bun_{\cG}(\cZ)$,  and its restriction $(\cE,t)_{K}$ to $X_{K}$ lifts to $\cF_{K}: \Spec K\to \Bun_{\cH}$ (an $\cH$-torsor over $X_{K}$). We would like to check that there exists a finite extension $K'/K$, with $R'=\cO_{K'}$, and an $\cH$-torsor $\cF_{R'}$ over $X_{R'}$, viewed as a map $\Spec R'\to \Bun_{\cH}$  such that the following diagram is commutative
\begin{equation*}
\xymatrix{ \Spec K' \ar[r]\ar[d] & \Spec K \ar[r]^{\cF_{K}}\ar[d] & \Bun_{\cH}\ar[d]^{c^{\Bun}_{\cV}}\\
\Spec R' \ar@{-->}[urr]^{\cF_{R'}}\ar[r] &\Spec R \ar[r]^{(\cE,t)} & \Bun_{\cG}(\cZ)}
\end{equation*}
Moreover, any finite extension $K'/K$, the dotted arrow (together with $2$-isomorphisms making the diagram commutative) should be unique up to unique isomorphism.

For this we may replace $K$ by $C$, the completion of an algebraic closure of $K$, and replace $R$ by $\cO_{C}$, and show the existence and uniqueness of the dotted arrow in the following diagram 
\begin{equation*}
\xymatrix{ \Spec C\ar[rr]^{\cF_{C}}\ar[d] & & \Bun_{\cH}\ar[d]^{c^{\Bun}_{\cV}}\\
\Spec \cO_{C}\ar@{-->}[urr]^{\cF}\ar[rr]^{(\cE,t)} & &\Bun_{\cG}(\cZ)
}
\end{equation*}

Since \eqref{cz} is an isomorphism over $U$, $(\cE,t)$ gives an $\cH$-reduction $\cF_{U}$ of $\cE|_{U_{\cO_{C}}}$. We already have an $\cH$-torsor $\cF_{C}$ over $X_{C}$ and by the commutation of the square above, it coincides with $\cF_{U}$ over $U_{C}$. Therefore we have an $\cH$-torsor $\cF^{\c}$ over $U_{\cO_{C}}\cup X_{C}=X_{\cO_{C}}\bs (X\bs U)$ (where $X\bs U$ is identified with a subset of the special fiber of $X_{\cO_{C}}$). We only need to extend  $\cF^{\c}$  to an $\cH$-torsor $\cF$ over $X_{\cO_{C}}$ (the diagram above will then be commutative, for the datum of $t$ is determined by its restriction to $U_{\cO_{C}}$).   

For each $x\in X\bs U$, recall the completed local ring $\cO_{x}$ and its fraction field $F_{x}$. Let $D_{x}=\Spec \cO_{x}$ and $D^{\times}_{x}=\Spec F_{x}$. For any  $\ov\FF_{q}$-algebra $A$, we denote $D_{x,A}=\Spec (\cO_{x}\wh\ot_{\ov\FF_{q}} A)$ and $D^{\times}_{x,A}=\Spec (F_{x}\wh\ot_{\ov\FF_{q}}A)$.  It suffices to show that the $\cH$-torsor  $\cF^{\c}|_{D_{x,\cO_{C}}\bs \{x\}}$ extends to $D_{x,\cO_{C}}$, and the extension is unique up to a unique isomorphism. Then use Beauville-Laszlo gluing  to get the existence and uniqueness of the global extension to $X_{\cO_{C}}$.

By a deep result of Ansch\"utz \cite[paragraph after Theorem 1.1]{An}, for $\cH$ a Bruhat-Tits group scheme, any $\cH$-torsor on $D_{x,\cO_{C}}\bs \{x\}$ is trivial. We fix a trivialization of $\cF_{U}|_{D^{\times}_{x,\cO_{C}}}$. Then extensions of $\cF_{U}|_{D^{\times}_{x,\cO_{C}}}$ to $D_{x,\cO_{C}}$ are parametrized by $\cO_{C}$-points of the affine flag variety $\Fl_{\cH,x}:=L_{x}\cH/L_{x}^{+}\cH$. The trivialization of $\cF_{U}|_{D^{\times}_{x,\cO_{C}}}$ induces a trivialization of $\cE|_{D^{\times}_{x,\cO_{C}}}$, so that extensions of $\cE|_{D^{\times}_{x,\cO_{C}}}$ to $D_{x,\cO_{C}}$ are parametrized by $\cO_{C}$-points of $\Fl_{\cG,x}=L_{x}\cG/L_{x}^{+}\cG$, and $\cE|_{D_{x,\cO_{C}}}$ gives a particular $\cO_{C}$-point of $\Fl_{\cG,x}$. On the other hand, $\cF_{C}|_{D_{x,C}}$ gives a $C$-point of $\Fl_{\cH,x}$ so we have a commutative diagram
\begin{equation*}
\xymatrix{\Spec C \ar[rr]^{\cF_{C}|_{D_{x,C}}}\ar[d] && \Fl_{\cH,x}\ar[d]\\
\Spec \cO_{C}\ar[rr]^{\cE|_{D_{x,\cO_{C}}}}\ar@{-->}[urr] && \Fl_{\cG,x}}
\end{equation*}
By \cite{Rich}, $\Fl_{\cH,x}$ is ind-proper, therefore the dotted arrow above exists and is unique. This shows the existence and uniqueness of $\cF|_{D_{x,\cO_{C}}}$ for each $x\in X\bs U$, hence the existence and uniqueness of $\cF$ itself.

%Assuming this claim, and we rename  $R'$ by $R$, $K'$ by $K$, we are in the situation where $\cF_{U}|_{D^{\times}_{x,R}}$ are trivial. Choose such trivializations, then extensions of $\cF_{U}|_{D^{\times}_{x,R}}$ to $D_{x,R}$ are parametrized by $R$-points of the affine flag variety $L_{x}\cH/L_{x}^{+}\cH$. By assumption, $\cH|_{D_{x}}$ is a Bruhat-Tits group scheme, then $L_{x}\cH/L_{x}^{+}\cH$ is ind-proper over $k(x)$ by \cite[]{Rich}. The restriction of $\cF_{K}$ to $D_{x,K}$ gives a $K$-point of $L_{x}\cH/L_{x}^{+}\cH$, which then extends unique to an $R$-point. This gives an extension of $\cF_{U}|_{D^{\times}_{x,R}}$ to $D_{x,R}$. By Beauville-Laszlo gluing with $\cF_{U}$,  we get an $\cH$-torsor $\cF_{R}$ on $X_{R}$. By construction, the restriction of $\cF_{R}$ to $X_{K}$ is isomorphic to $\cF_{K}$. This completes the construction of $\cF_{R}$ and proves that $c^{\Bun}_{\cV}$ is proper. 
  
Knowing that $c^{\Bun}_{\cV}$ is proper, it remains to show that geometric fibers of $c^{\Bun}_{\cV}$ are either empty or a (reduced) geometric point. Let $K\supset \ov\FF_{q}$ be an algebraically closed field, and let $(\cE,t)$ be a $K$-point of $\Bun_{\cG}(\cZ)$. As above, $(\cE,t)_{U_{K}}$ determines an $\cH$-reduction $\cF_{U}$ over $U_{K}$. Upon choosing a trivialization of $\cF_{U}|_{D^{\times}_{x,K}}$ for each $x\in X\bs U$, extensions of $\cF_{U}$ to $X_{K}$ are parametrized by $\prod_{x\in X\bs U}\Fl_{\cH,x}(K)$. On the other hand, extensions of $\cE|_{U_{K}}$ to $X_{K}$  are parametrized by $\prod_{x\in X\bs U}\Fl_{\cG,x}(K)$. Therefore, the fiber $c^{\Bun,-1}_{\cV}(\cE,t)$, as a $K$-scheme, is the fiber of 
\begin{equation}\label{phx}
\prod_{x\in X\bs U}\Fl_{\cH,x}\to \prod_{x\in X\bs U} \Fl_{\cG,x}
\end{equation}
over the given $K$-point of $\prod_{x\in X\bs U} \Fl_{\cG,x}$ given by $\cE|_{D_{x,K}}$. Since $\cH\incl \cG$ is a closed embedding, so is \eqref{phx}, therefore the fiber of \eqref{phx} over a $K$-point is either $\Spec K$ or empty. This finishes the proof. 

%Let $L_{x}\cE$ be the loop space of $\cE$ at $x$ (as an ind-scheme over $S$), i.e., the functor $R\mapsto \cE(R\wh\ot F_{x})$ where $\Spec R$ is over $S$. Let $L^{+}_{x}\cG$ be the arc group $R\mapsto \cG(R\wh\ot\cO_{x})$ (a pro-algebraic group over the residue field $k_{x}$).  Then the fpqc quotient $L_{x}\cE/(L_{x}^{+}\cG\times S)$ parametrizes extensions of the $\cG$-torsor $\cE|_{D^{\times}_{x}\wh\times S}$ to $D_{x}\wh\times S$. Note that, since $\cG|_{D_{x}}$ is contained in a parahoric subgroup of $\cG|_{D^{\times}_{x}}$,  $L_{x}\cE/(L_{x}^{+}\cG\times S)$ is representable by an ind-scheme. Similar remarks apply to $\cH$ in place of $\cG$, and $\cF$ in place of $\cE$. We have a natural map of ind-schemes over $S$
%\begin{equation}
%\ph_{x}: L_{x}\cF/(L^{+}_{x}\cH\times S)\to L_{x}\cE/(L^{+}_{x}\cG\times S).
%\end{equation}
%This is a closed embedding: replacing $S$ be an \'etale covering, we may assume $\cE$ and $\cF$ are trivial over $D^{\times}_{x}\wh\times S$, and the closedness of $\ph_{x}$ follows from the closedness of $\cH$ in $\cG$.
%The $\cG$-torsor $\cE|_{D_{x}\wh\times S}$ gives a section $\s_{x}: S\to L_{x}\cE/(L^{+}_{x}\cG\times S)$.  By construction, we have a Cartesian diagram
%\begin{equation}
%\xymatrix{S'\ar[r]\ar[d] & \prod_{x\notin U}L_{x}\cF/(L^{+}_{x}\cH\times S)\ar[d]^{(\ph_{x})}\\
%S\ar[r]^-{(\s_{x})} & \prod_{x\notin U}L_{x}\cE/(L^{+}_{x}\cG\times S)}
%\end{equation}
%Since each $\ph_{x}$ is a closed embedding, so is its base change $S'\to S$. 
\end{proof}

\section{Closedness of special cycles}
In this section, the base field $k=\FF_{q}$.  We will prove Theorem \ref{th:main} in this section. %Replacing $\cG$ by its inner form $\uAut_{\cG}(\cE_{0})$ does not change $\Bun_{\cG}$ and $\Sht^{r}_{\cG}$. So in rest of the section we assume the pseudo-homomorphism actually comes from a group monomorphism $i:\cH\to \cG$ (i.e., $\cE_{0}$ is a trivial $\cG$-torsor).

\subsection{Shtukas with sections} Let $\cW\in \Rep(\cG)$. Consider the  ind-stack $\Sht^{r}_{\cG}(\cW)$ whose $S$-points are tuples $(x_{\bu}, \cE_{\bu}, f_{\bu}, \io, t_{\bu})$ where
\begin{itemize}
\item $(x_{\bu},\cE_{\bu},f_{\bu}, \io)\in \Sht^{r}_{\cG}(S)$;
\item For $0\le i\le r$, $t_{i}$ is a section of $\cW_{\cE_{i}}$
\end{itemize}
such that the following diagram is commutative 
\begin{equation}\label{map from F}
\xymatrix{\cO_{X\times S}\ar[d]^{t_{0}}\ar@{=}[r] & \cdots \ar@{=}[r] & \cO_{X\times S} \ar@{=}[r]\ar[d]^{t_{r}} & \cO_{X\times S}\ar[d]^{{}^{\t}t_{0}}\\
\cW_{\cE_{0}}\ar@{-->}[r]^-{f_{1}} & \cdots\ar@{-->}[r]^-{f_{r}} &  \cW_{\cE_{r}}\ar[r]^-{\io}_-{\sim} & \cW_{{}^{\t}\cE_{0}}={}^{\t}(\cW_{\cE_{0}})}
\end{equation}
Define an open substack $\Sht^{r}_{\cG}(\cW)^{\c}\subset \Sht^{r}_{\cG}(\cW)$ whose $S$-points are those $(x_{\bu}, \cE_{\bu}, f_{\bu}, \io, t_{\bu})$ such that for any geometric point $s\in S$, the restriction of $t_{i}$ to $X\times \{s\}$ is nonzero for any $0\le i\le r$.

\begin{prop}\label{p:sect fin} Let $\cG$ be a smooth affine group scheme over $X$. For $\cW\in\Rep(\cG)$, the forgetful map
\begin{equation*}
\Forg_{\cW} :\Sht^{r}_{\cG}(\cW)^{\c}\to \Sht^{r}_{\cG}
\end{equation*}
is schematic, finite and unramified.
\end{prop}
\begin{proof} From the definition it is clear that $\Forg_{\cW}$ is schematic. We first show that $\Forg_{\cW}$ is proper.   For this we introduce a projectivized version of $\Sht^{r}_{\cG}(\cW)^{\c}$. 

Let $\PP\Bun_{\cG}(\cW)$ be the moduli stack whose $S$-points are triples $(\cE, \cM,t)$ where $\cE$ is a $\cG$-torsor over $X\times S$, $\cM$ is a line bundle over $S$, and $t: \cO_{X}\boxtimes\cM\to \cW_{\cE}$ is a map of coherent sheaves on $X\times S$ that is nonzero on $X\times \{s\}$ for all geometric points $s\in S$.

We claim that $\PP\Bun_{\cG}(\cW)$ is the projectivization of a sheaf of graded algebra over $\Bun_{\cG}$. Indeed,  let $\cE_{\univ}$ be the universal $\cG$-torsor over $\Bun_{\cG}\times X$, and let $\cW_{\univ}=\cW_{\cE_{\univ}}$ (a vector bundle over $\Bun_{\cG}\times X$). Let $p_{\Bun}: \Bun_{\cG}\times X\to \Bun_{\cG}$ and $p_{X}: \Bun_{\cG}\times X\to X$ be the projections. Consider the coherent sheaf $\cA=\bR^{1}p_{\Bun*}(\cW^{\vee}_{\univ}\ot p_{X}^{*}\om_{X})$ on $\Bun_{\cG}$.  Let $\Proj(\Sym^{\bu}(\cA))$ (relative to $\Bun_{\cG}$) be the projective bundle of hyperplanes in fibers of $\cA$.  By definition, $\Proj(\Sym^{\bu}(\cA))(S)$ classifies triples $(\cE, \cM, \s)$ where $\cE$ is a $\cG$-torsor over $X\times S$, $\cM$ is a line bundle over $S$, and $\s$ is a surjective map of coherent sheaves on $S$
\begin{equation*}
\s: h_{\cE}^{*}\cA\surj \cM^{\vee}.
\end{equation*}
Here $h_{\cE}: S\to \Bun_{G}$ is given by $\cE$. Equivalently, $\s$ is the same datum as a map of complexes $\bL h^{*}_{\cE}\bR p_{\Bun*}(\cW^{\vee}_{\univ}\ot p_{X}^{*}\om_{X})[1]\to \cM^{\vee}$, hence by base change is the same as $\bR \pr_{S*}(\cW^{\vee}_{\cE}\ot \pr_{X}^{*}\om_{X})[1]\to \cM^{\vee}$ (we use $\pr_{S}$ and $\pr_{X}$ to denote the projections $X\times S\to S$ and $X\times S\to X$).  %By shrinking $S$ we may present $\bR p_{S*}(\cV^{\vee}_{\cE}\ot p_{X}^{*}(\cF\ot\om_{X}))[1]$ as a two term perfect complex $\cK_{1}\xr{d} \cK_{0}$ in degrees $-1$ and $0$. Then  the map $\s$ corresponds to a surjective map $\s': \cK_{0}\surj \cM^{\vee}$ such that $\s'\c d=0: \cK_{1}\to \cM^{\vee}$. 
Taking relative Serre dual,  $\s^{\vee}$ is a map $\cM\to \bR \pr_{S*}(\pr_{X}^{*}\cW_{\cE})$ that is nonzero on $X\times \{s\}$ for all geometric points $s\in S$. Equivalently this is a map $t: \cO_{X}\boxtimes \cM\to \cW_{\cE}$ as in the definition of $\PP\Bun_{\cG}(\cW)$. This proves $\PP\Bun_{\cG}(\cW)\cong \Proj(\Sym^{\bu}(\cA))$. In particular, $\PP\Bun_{\cG}(\cW)\to \Bun_{\cG}(\cW)$ is proper.

%the complex $\cK_{0}^{\vee}\xr{d^{\vee}}\cK_{1}^{\vee}$ (in degrees $0$ and $1$) is quasi-isomorphic to $\bR p_{S*}(p_{X}^{*}\cF^{\vee}\ot \cV_{\cE})$. Then $\s$ is the same datum as a map $t: \cM\to \cK_{0}^{\vee}$ that is locally a direct summand, such that $d^{\vee}\c t=0: \cM^{\vee}\to \cK_{1}^{\vee}$. Equivalently this is the same data as a map $t: \cF\boxtimes \cM\to \cV_{\cE}$ that is nonzero on $X\times \{s\}$ for all geometric points $s\in S$. This proves the claim.

Similarly define $\PP\Hk^{r}_{\cG}(\cW)$ to classify $(x_{\bu}, \cE_{\bu}, f_{\bu}, \cM, t_{\bu})$ where $(x_{\bu}, \cE_{\bu}, f_{\bu})\in \Hk^{r}_{\cG}(\cW)(S)$, $\cM$ is a line bundle on $S$ and $t_{i}: \cO_{X}\bt\cM\to \cW_{\cE_{i}}$ ($0\le i\le r$), compatible with the modifications $f_{\bu}$ and nonzero when restricted to $X\times \{s\}$ for all geometric points $s$.

We claim that the forgetful map $\PP\Hk^{r}_{\cG}(\cW)\to\Hk^{r}_{\cG}$ is proper. Indeed, $\PP\Hk^{r}_{\cG}(\cW)$ is closed in $\Hk^{r}_{\cG}\times_{p_{0}, \Bun_{\cG}}\PP\Bun_{\cG}(\cW)$ because $t_{0}$ determines all $t_{i}$ for $i\ge1$, and the existence of $t_{i}$ means the map $t_{0}: \cO_{X}\bt\cM|_{X\times S\bs \cup_{j\le i}\G(x_{j})}\to \cW_{\cE_{i}}|_{X\times S\bs \cup_{j\le i}\G(x_{j})}$  extends to the whole $X\times S$, which imposes a closed condition on $t_{0}$.

Finally we define $\PP\Sht^{r}_{\cG}(\cW)$ by the Cartesian square
\begin{equation}\label{def PSht}
\xymatrix{\PP\Sht^{r}_{\cG}(\cW)\ar[rr]\ar[d] && \PP\Hk^{r}_{\cG}(\cW)\ar[d]^{(p_{0},p_{r})}\\
\PP\Bun_{\cG}(\cW)\ar[rr]^-{(\id, \Fr)} && \PP\Bun_{\cG}(\cW)\times\PP\Bun_{\cG}(\cW)}
\end{equation}

The Cartesian square \eqref{def PSht} maps to the Cartesian square \eqref{def Sht} defining $\Sht^{r}_{\cG}$, and the maps $\PP\Hk^{r}_{\cG}(\cW)\to \Hk_{\cG}$ and $\PP\Bun_{\cG}(\cW)\to \Bun_{\cG}$ are both proper. Therefore the map $\PP\Sht^{r}_{\cG}(\cW)\to \Sht^{r}_{\cG}$ is proper. We have a factorization
\begin{equation*}
\Forg_{\cW}: \Sht^{r}_{\cG}(\cW)^{\c}\xr{\pi_{\cW}} \PP\Sht^{r}_{\cG}(\cW)\to \Sht^{r}_{\cG}
\end{equation*}
where the first map sends $(x_{\bu}, \cE_{\bu},f_{\bu}, \io, t_{\bu})$ to $(x_{\bu}, \cE_{\bu},f_{\bu}, \io, \cM=\cO_{S}, t_{\bu})$. To show $\Forg_{\cW}$ is proper, it remains to show that $\pi_{\cW}$ is proper. We claim that $\pi_{\cW}$ is a $\FF_{q}^{\times}$-torsor. Clearly $\FF_{q}^{\times}$ acts on $\Sht^{r}_{\cG}(\cW)^{\c}$ by scaling $t_{i}$, and $\pi_{\cW}$ is invariant under this action. Take an $S$-point $\xi=(x_{\bu}, \cE_{\bu}, f_{\bu}, \io,\cM, t_{\bu})$ of $\PP\Sht^{r}_{\cG}(\cW)$. By shrinking $S$ we may assume $\cM=\cO_{S}$. Then ${}^{\t}t_{0}=ct_{r}: \cO_{X\times S}\to {}^{\t}(\cW_{\cE_{0}})$ for a unique invertible function  $c\in \upH^{0}(S,\cO_{S}^{\times})$. Then the fiber of $\pi_{\cW}$ over this $\xi$ is the subscheme of $\GG_{m,S}$ classifying $b\in \cO^{\times}_{S}$ such that $b^{q}=cb$ (for such $b$, $t'_{i}=bt_{i}$ makes the rightmost square of the diagram \eqref{map from F} commutative). Therefore $\pi_{\cW}$ is a $\FF_{q}^{\times}$-torsor.

It remains to show that $\Forg_{\cW}$ is unramified. For this we fix any algebraically closed field $K\supset k$ and a $K$-point $(x_{\bu}, \cE_{\bu},f_{\bu},\io)\in \Sht^{r}_{\cG}(K)$, and show that its fiber under $\Forg_{\cW}$ is finite and reduced over $K$. Let $U_{i}=\upH^{0}(X_{K}, \cW_{\cE_{i}})$ for $0\le i\le r$. Then we have a Frobenius semilinear isomorphism $\phi: U_{0}\isom U_{r}$ induced by $\io$.  All $U_{i}$ lie in the same $F_{K}=K(X)$-vector space  $U_{\y}$ which is the generic fiber of $\cW_{\cE_{0}}$, and $\phi$ extends to an endomorphism of $U_{\y}$. The fiber $\Forg_{\cW}^{-1}(x_{\bu},\cE_{\bu},f_{\bu},\io)$ is the set of $t\in \cap_{i=0}^{r}U_{i}$ such that  $\phi(t)=t$. Let $\Om\subset U_{0}$ be the largest $K$-subspace that is stable under $\phi$. Then $\phi|_{\Om}$ gives a descent datum to a $\FF_{q}$-vector space $\Om_{0}=\Om^{\phi}$ such that $\Om_{0}\ot_{\FF_{q}}K=\Om$. Then $\Om_{0}$ is a finite-dimensional $\FF_{q}$-vector space and $\Forg_{\cW}^{-1}(x_{\bu},\cE_{\bu},f_{\bu},\io)=\Om_{0}$. This proves that $\Forg_{\cW}$ is unramified.
\end{proof}

\subsection{Proof of Theorem \ref{th:main}}
Choose $\cV\in \Rep(\cG)$ so that $b_{\cV}$ is a closed embedding over $U$. Let $\cW=(\cV^{\cH})^{\vee}\ot \cV$.  We factorize the map $\th^{r}: \Sht^{r}_{\cH}\to \Sht^{r}_{\cG}$ as follows
\begin{equation}\label{thr factor}
\xymatrix{\th^{r}: \Sht^{r}_{\cH}\ar[r]^-{b^{\Sht,\c}_{\cV}} & \Sht^{r}_{\cG}(\cW)^{\c} \ar[rr]^-{\Forg_{\cW}} & & \Sht^{r}_{\cG}}
\end{equation}
The map $b^{\Sht,\c}_{\cV}$ sends $(x_{\bu}, \cF_{\bu}, f'_{\bu}, \io')\in \Sht^{r}_{\cH}(S)$ to the tuple $(x_{\bu}, \cE_{\bu}=\cF_{\bu}\times^{\cH}\cG, f_{\bu}=f'_{\bu}\times\id_{\cG}, \io'=\io\times\id_{\cG}, t_{\bu}=\{t_{i}\}_{0\le i\le r})$ where $t_{i}: \cO_{X\times S}\to \cW_{\cE_{i}}=(\cV^{\cH})^{\vee}\ot_{\cO_{X}} \cV_{\cE_{i}}$ is induced from the canonical map $\cV^{\cH}\bt \cO_{S}=(\cV^{\cH})_{\cF_{i}}\to \cV_{\cF_{i}}=\cV_{\cE_{i}}$.

The last map $\Forg_{\cW}$ is schematic, finite and unramified by Proposition  \ref{p:sect fin}. Below we will show that $b^{\Sht,\c}_{\cV}$ is a closed embedding, which then implies that $\th^{r}$ is schematic, finite and unramifed, as desired.
%
%In the above sequence of maps,  $\Sht^{G-\mu}_{H}$ is the stack whose $S$-points classify $(x_{\bu}=\{x_{i}\}_{0\le i\le r}, \cE_{H,\bu}=\{\cE_{H,i}\}_{0\le i\le r}, f'_{\bu}=\{f'_{i}\}_{1\le i\le r}, \io)$ where $x_{i}:S\to X$, $\cE_{H,i}$ are $H$-torsors over $X\times S$, and letting $\cE_{G,i}$ be the induced $G$-torsor from $\cE_{H,i}$, the maps $f'_{i}: \cE_{G,i-1}|_{X\times S\bs \G(x_{i})}\isom \cE_{G,i}|_{X\times S\bs \G(x_{i})}$ bounded  by $\mu_{i}$ near $\G(x_{i})$ (for $1\le i\le r$)  and $\io: \cE_{H,r}\isom {}^{\t}\cE_{H,0}$. The map $i$ is the canonical one that keeps same data $(\x_{\bu}, \cE_{H,\bu},\io)$, and letting $f_{i}$ be the map induced from the modifications $f_{i}: \cE_{H,i-1}|_{X\times S\bs \G(x_{i})}\isom \cE_{H,i}|_{X\times S\bs \G(x_{i})}$. It is clear that $i$ is a closed embedding.

To show $b^{\Sht,\c}_{\cV}$ is a closed embedding, it suffices to prove that the map $b^{\Sht}_{\cV}: \Sht^{r}_{\cH}\to \Sht^{r}_{\cG}(\cW)$ is a closed embedding. For this we introduce the ind-stack $\Hk^{r}_{\cG}(\cW)$ to classify $(x_{\bu}, \cE_{\bu}, f_{\bu}, t_{\bu})$ as in $\Sht^{r}_{\cG}(\cW)$ satisfying the same conditions, except that there is  no $\io$.
Then we have a Cartesian square
\begin{equation}\label{ShtGV}
\xymatrix{    \Sht^{r}_{\cG}(\cW)      \ar[r]\ar[d] & \Hk^{r}_{\cG}(\cW)\ar[d]^{(p_{0},p_{r})} \\
\Bun_{\cG}(\cW)\ar[r]^-{(\id,\Fr)} & \Bun_{\cG}(\cW)\times \Bun_{\cG}(\cW)}
\end{equation}

%
%For this we introduce $\Hk^{G-\mu}_{H}$ to classify $(x_{\bu}, \cE_{H,\bu}, f'_{\bu})$ as in $\Sht^{G-\mu}_{H}$ (without $\io$). Then $\Sht^{G-\mu}_{H}$ fits into a Cartesian diagram
%\begin{equation}\label{ShtH}
%\xymatrix{   \Sht^{G-\mu}_{H}\ar[r]\ar[d] & \Hk^{G-\mu}_{H}\ar[d]^{(p_{0},p_{r})}        \\
%\Bun_{H}\ar[r]^-{(\id,\Frob)} & \Bun_{H}\times \Bun_{H}}
%\end{equation} 
%Here $p_{i}:\Hk^{G-\mu}_{H}\to  \Bun_{H}$ records $\cE_{H,i}$ for $0\le i\le r$.

Now there is a canonical map from the diagram defining $\Sht^{r}_{\cH}$ (replacing $\cG$ by $\cH$ in \eqref{def Sht}) to the diagram \eqref{ShtGV} by applying the basic map $b^{\Bun}_{\cV}$ to all the $\cH$-torsors. We denote the map on the upper right corner by $b^{\Hk}_{\cV}: \Hk^{r}_{\cH}\to \Hk^{r}_{\cG}(\cW)$; the resulting map on the upper left corner is the map $b^{\Sht}_{\cV}$ we introduced before. By Proposition  \ref{l:bV par}, $b^{\Bun}_{\cV}$ is a closed embedding. In the next lemma we shall show that $b^{\Hk}_{\cV}$ is also a closed embedding. This implies that  $b^{\Sht}_{\cV}$ is also  a closed embedding, being the fiber product of closed embeddings over a closed embedding. This finishes the proof. 
\qed

\begin{lemma}\label{l:bV Hk}
The map $b^{\Hk}_{\cV}$ is a closed embedding.
\end{lemma}
\begin{proof}
We have a commutative diagram
\begin{equation}\label{HkW}
\xymatrix{\Hk^{r}_{\cH}\ar[r]^-{b^{\Hk}_{\cV}}\ar[d]^{(p_{i})_{0\le i\le r}} & \Hk^{r}_{\cG}(\cW)\ar[d]^{(p_{i})_{0\le i\le r}}\\
(\Bun_{\cH})^{r+1}\ar[r]^-{\prod b^{\Bun}_{\cV}} & \Bun_{\cG}(\cW)^{r+1}}
\end{equation}
We claim that this diagram is Cartesian, which would prove that $b^{\Hk}_{\cV}$ is a closed embedding, since $b^{\Bun}_{\cV}$ is by Proposition  \ref{l:bV par}.

%Then we have a factorization of $b^{\Hk}_{\cV}$:
%\begin{equation}
%\Hk^{r}_{\cH}\xr{\a} \Hk^{r, \cW}_{\cH}\xr{\b} \Hk^{r}_{\cG}(\cW).
%\end{equation}
%Now $\b$ is the base change of a product of $b^{\Bun}_{\cV}$, which is a closed emebdding by Proposition \ref{p:bV par}, therefore $\b$ is a closed  embedding. It remains to show that $\a$ is also a closed embedding.

Let $\Hk^{r, \cW}_{\cH}$ be the fiber product $\Hk^{r}_{\cG}(\cW)$ and $(\Bun_{\cH})^{r+1}$ over $\Bun_{\cG}(\cW)^{r+1}$ using the maps in the above diagram.  By definition, an $S$-point of $\Hk^{r, \cW}_{\cH}$ is a tuple $(x_{\bu}, \cF_{\bu}, f_{\bu}, t_{\bu})$ where, denoting $\cE_{i}=\cF_{i}\times^{\cH}_{X}\cG$, $f_{i}: \cE_{i-1}|_{X\times S\bs \G(x_{i})}\isom \cE_{i}|_{X\times S\bs \G(x_{i})}$ is a $\cG$-isomorphism and intertwines $t_{i-1}$ and $t_{i}$. We would like to show that $f_{i}$ comes from a unique isomorphism of $\cH$-torsors  $f'_{i}: \cF_{i-1}|_{X\times S\bs \G(x_{i})}\isom \cF_{i}|_{X\times S\bs \G(x_{i})}$.  

Let $Y=X\times S\bs \G(x_{i})$. 
Write $\cE$ for the $\cG$-torsor $\cE_{i-1}|_{Y}\cong \cE_{i}|_{Y}$ (identified via $f_{i}$) over $Y$. The algebraic space $\cE/\cH$ over $Y$ classifies reductions of $\cE$ to $\cH$ (over test $Y$-schemes). In particular, $\cF_{i-1}|_{Y}$ and $\cF_{i}|_{Y}$ give two $\cH$-reductions of $\cE$ to $\cH$ hence define two sections $\xi_{i-1}, \xi_{i}: Y\to \cE/\cH$. The map $b_{\cV}$ induces a map 
\begin{equation*}
\b: \cE/\cH\to (\cV^{\cH})\ot_{\cO_{X}} \cV_{\cE}
\end{equation*}
where $\cV_{\cE}$ is the vector bundle over $Y$ associated to $\cE$ and $\cV$. The compositions $\b\c\xi_{i-1}$ and $\b\c\xi_{i}$ are the sections $t_{i-1}|_{Y}$ and $t_{i}|_{Y}$ respectively. By assumption, $t_{i-1}|_{Y}$ and $t_{i}|_{Y}$ are identified via $f_{i}$, hence
\begin{equation}\label{bx}
\b\c\xi_{i-1}=\b\c\xi_{i}.
\end{equation}
Recall that $\cV$ is chosen so that $b_{\cV}$ is a closed embedding over the open curve $U$, hence $\b$ is a closed embedding over $Y'=(U\times S) \cap Y$ (inside $X\times S$). Then \eqref{bx} implies that $\xi_{i-1}|_{Y'}=\xi_{i}|_{Y'}$. Since $Y'$ is dense in $Y$ and $\cE/\cH$ is separated over $Y$ (for $\cH\incl \cG$ is closed), we must have $\xi_{i-1}=\xi_{i}$, i.e., the equality of the $\cH$-reductions $\cF_{i-1}|_{Y}$ and $\cF_{i}|_{Y}$ of $\cE=\cE_{i-1}|_{Y}\cong \cE_{i}|_{Y}$. This shows that \eqref{HkW} is Cartesian and finishes the proof.

\end{proof}

\end{document}